\documentclass[leqno,12pt]{article}

\usepackage{amsmath}
\usepackage{amscd}
\usepackage{amsopn}
\usepackage{amsthm}
\usepackage{amsfonts,amssymb}
\usepackage{amsfonts,bbm}
\usepackage{latexsym}

\usepackage{color}

\makeatletter
\@addtoreset{equation}{section}
\makeatother

\newtheorem{lemma}{LEMMA}[section]
\newtheorem{proposition}[lemma]{PROPOSITION}
\newtheorem{corollary}[lemma]{COROLLARY}
\newtheorem{theorem}[lemma]{THEOREM}
\newtheorem{remark}[lemma]{REMARK}

\newtheorem{example}[lemma]{EXAMPLE}

\newcommand{\real}{\mathbbm{R}}
\newcommand{\nat}{\mathbbm{N}}

\renewcommand{\a}{\alpha}
\renewcommand{\b}{\beta}

\newcommand{\g}{\gamma}

\newcommand{\vp}{\varphi}
\newcommand{\ve}{\varepsilon}

\newcommand{\reald}{{\real^d}}

\newcommand{\on}{\quad\text{ on }}
\newcommand{\und}{\quad\mbox{ and }\quad}
\newcommand{\inv}{^{-1}}
\newcommand{\ov}{\overline}

\newcommand{\V}{\mathcal V}  
\newcommand{\W}{\mathcal W}  
  
\newcommand{\N}{\mathcal N}  
\newcommand{\C}{\mathcal C}  

\newcommand{\F}{\mathcal F}

\renewcommand{\H}{{\mathcal H}}
\newcommand{\B}{\mathcal B}

\newcommand{\M}{\mathcal M}
\newcommand{\A}{\mathcal A}

\newcommand{\U}{{\mathcal U}}
\newcommand{\J}{{\mathcal J}}

\newcommand{\supp}{\operatorname*{supp}}
\newcommand{\fliminf}{\mbox{f-}\!\liminf}

\newcommand{\itemframe}%
{\setlength{\parskip}{10pt}\begin{enumerate} \setlength{\topsep}{10pt}%
\setlength{\itemsep}{15pt}\setlength{\parsep}{5pt}}

\newcommand{\vx}{\ve_x}

\newcommand{\Px}{\mathcal P(X)}

\newcommand{\uc}{{U^c}}

\newcommand{\vc}{{V^c}}

\newcommand{\h}{{}^\ast\H}

\newcommand{\schluss}{\end{frame}\end{document}}

\newcommand{\bu}{\B^\ast}
\newcommand{\bup}{\widetilde{\B^\ast}}
\newcommand{\bp}{\widetilde{\B }}

 \title{Nearly hyperharmonic functions\\ and Jensen measures}
\author {Wolfhard Hansen and  Ivan Netuka}

\begin{document}

\maketitle

  \begin{abstract}
Let $(X,\H)$ be a $\mathcal P$-harmonic space and assume for simplicity
that constants are harmonic. Given 
a numerical function $\vp$ on $X$ which is locally lower bounded  let 
\begin{equation*} 
     J_\vp(x):=\sup\{\int^\ast \vp\,d\mu(x)\colon \mu\in \mathcal J_x(X)\},
\qquad x\in X,
\end{equation*} 
where $\mathcal J_x(X)$ denotes the set of all  Jensen measures $\mu$ for $x$, that is,
$\mu$ is a~compactly supported measure on $X$ satisfying $\int u\,d\mu\le u(x)$
for every hyperharmonic function on $X$. The main purpose of the paper is
to show that, 
assuming quasi-universal measurability of $\vp$, 
the function  $J_\vp$ is the smallest nearly hyperharmonic function majorizing 
$\vp$ and that $J_\vp=\vp\vee \hat J_\vp$, where $\hat J_\vp$ is the lower semicontinuous 
regularization of $J_\vp$. So, in particular, $J_\vp$ turns out to be at least ``as measurable as'' $\vp$.

This improves recent results, where the axiom of polarity was assumed. The preparations
about nearly hyperharmonic functions on balayage spaces are closely related to the study
of strongly supermedian functions triggered by J.-F.\ Mertens more than forty years ago.


\end{abstract}

\section{Representing measures for positive hyperharmonic functions}

Let $(X,\W)$ be a balayage space  ($X$ a locally compact space with countable base   and $\W$ the set of 
positive hyperharmonic functions on $X$, see \cite{BH} or \cite{H-course}). The fine topology on $X$
(it is finer than the initial topology) is the coarsest topology such that all functions in $\W$ are continuous.
Let~$\C(X)$ denote the set of all continuous real functions on $X$, and let us fix a strictly positive 
function $u_0\in \W\cap \C(X)$.
Further, let $\Px$
be the set of all continuous real potentials, that is, functions $p\in\W\cap\C(X)$ such that 
$p/v$ vanishes at infinity for some strictly positive $v\in \W\cap \C(X)$. We shall say that 
a numerical function~$\vp$ on~$X$ is $\mathcal P$-bounded, if $|\vp|\le p$ for some $p\in \Px$;
the set of all $\mathcal P$-bounded functions in~$\C(X)$ will be denoted by $\C_{\mathcal P}(X)$.
For every numerical function $\vp$ on $X$, let $\hat\vp$ denote 
its lower semicontinuous regularization, 
that is, $\hat \vp(x):=\liminf_{y\to x}\vp(y)$ for every  $x\in X$. If $\V$ is a subset of $\W$ and $v:=\inf \V$,
then $\hat v\in \W$ and $\hat v(x)=\hat v^f(x):=\fliminf_{y\to x} v(x)$, $x\in X$ (lower limit with respect to the fine topology).

We recall that, for every numerical function $\vp\ge 0$,   a reduced function~$R_\vp$ and a~swept function $\hat R_\vp$ are  defined by
 \begin{equation}\label{def-red}
      R_\vp:=\inf\{u\in \W\colon u\ge \vp\} \und \hat R_\vp:=\widehat {R_\vp}.
\end{equation} 
In particular,  we have $R_v^A:=R_{v1_A}\le v$ and $\hat R_v^A:=\hat R_{v1_A}\le R_v^A$ for $A\subset X$ and $v\in \W$, 
which leads to reduced measures~$\vx^A$ and swept measures $\hat\ve_x^A$, $x\in X$, characterized by $\int v\,d\vx^A=R_v^A(x)$
and $\int v\,d\hat\ve_x^A=\hat R_v^A(x)$, $v\in \W$.
Let us observe that $\vx^A =\hat\ve_x^A$ for every $x\in A^c$, since $\hat R_v^A=R_v^A$ on $A^c$
(see \cite[VI.2.4]{BH}).
If $x\in A$, then  $\vx^A=\vx$ and, by~\cite[VI.9.2]{BH},
$\hat \ve_x^A=\hat\ve_x^A(\{x\})\vx+(1-\hat\ve_x^A(\{x\})\vx^{A\setminus \{x\}}$.

For every $x\in X$, let $\M_x(\W)$ denote the convex set of all 
\emph{representing measures for $x$ with respect to $\W$},
that is, (positive Radon) measures $\mu$ on $X$ such that, for every $w\in\W$,
\begin{equation}\label{def-mxw} 
\int w\,d\mu\le w(x).
\end{equation} 
Since every function in $\W$ is an increasing limit of a sequence in $\Px$, (\ref{def-mxw}) holds for functions 
in $\W$, if it holds for functions in $\Px$.  
Let $\B$, $\bu$ respectively  denote the   
 $\sigma$-algebra of all Borel, ($\B$-)universally measurable  sets in $X$.
By \cite[VI.12.5, 2.2, 4.3, 4.4]{BH}, 
\begin{equation}\label{ext-M}
        E:=\{\vx^A\colon A\subset X\}=\{\vx\}\cup \{\vx^A\colon A\in \B,\  A \mbox{ finely closed, } x\notin A\}
\end{equation} 
is the set of extreme points of $\M_x(\W)$. The set 
 $\M_x(\W)$ is  weak$^\ast$-compact, that is, for every sequence $(\mu_m)$
in $\M_x(\W)$, there exists a subsequence $(\mu_{m_k})$ and $\mu\in\M_x(\W)$ 
such that $\lim_{k\to\infty} \int f\,d\mu_{m_k}=\int f\,d\mu$ for every $f\in \C_{\mathcal P}(X)$
(see \cite[VI.10.1]{BH}). So we know by Choquet's theorem that, for every $\mu\in\M_x(\W)$, there exists
 a~probability measure $\rho$  on $E$ such that, for every $f\in\C_{\mathcal P}(X)$,
\begin{equation}\label{int-rep}
\int f\,d\mu=\int (\int f d\nu)\,d\rho(\nu),
\end{equation} 
and then (\ref{int-rep}) holds for every Borel measurable function $f\ge 0$ on $X$.
(We might note that, for a given $\mu$, the measure $\rho$ does not have to be unique; see \cite{HN-square}).

By definition, 
 a subset $P$ of $X$ is polar if $\hat R_{u_0}^P=0$. Every polar set $P$ is contained in a polar set in $\B$
(see \cite[VI.2.2]{BH}). 
Let $\tilde \B,\bup$  denote the $\sigma$-algebra of all sets  $A$ in $X$    
for which there exists a set $B$ in $\B,  \bu$ respectively such that the symmetric difference  $A\bigtriangleup B$ is polar.

If  $\mu\in \M_x(\W)$, $x\in X$, then $\mu$ does not charge polar sets $P$ in $X\setminus \{x\}$.
Indeed, given $\ve>0$, there exists a function $w\in \W$ such that 
$w=u_0$ on $P$ and $w(x)<\ve$, and we have $\int_P^\ast u_0\, d\mu\le \int w\,d\mu\le w(x)<\ve$
(cf.\ \cite[Corollary 1.8]{cole-ransford-subharmonicity}), whence $\mu^\ast(P)=0$.                 
So we know that $ \bup$ is contained in the completions of $\B$ with respect to the measures  $\mu\in \M_x(\W)$, $x\in X$.

\section{Nearly hyperharmonic positive functions}

Let $\U_c$ denote the set of all relatively compact open sets in $X$ and let us say that a~positive numerical function 
$u$ on $X$ is \emph{nearly hyperharmonic} if
\begin{equation}\label{def-nearly}
    \int^\ast u\,d\vx^\vc\le u(x)\qquad \mbox{ for all } V\in \U_c\mbox{ and }x\in X.
\end{equation} 
This generalizes the definition given  for harmonic spaces in \cite[Section II.1]{bauer66} and \cite[p.\ 119]{Const}). 
 As for harmonic spaces we easily obtain the following.

\begin{proposition}\label{nearly-proper} 
The  set $\N^+$  
of all  nearly hyperharmonic positive functions on $X$ has the following properties:
\begin{itemize} 
\item[\rm (i)] 
$\N^+$ is a convex cone containing $\W$. 
\item[\rm (ii)] 
For every $u\in\N^+$, $\hat u=\hat u^f\in\W$.
\item[\rm (iii)] 
If $(u_m)$ is a sequence in $\N^+$ and $u_m\uparrow u$, then $u\in\N^+$ and $\hat u_m\uparrow \hat u$.
\item[\rm(iv)] For every subset $\V$ of $\N^+$, $\inf \V\in \N^+$. 
\end{itemize} 
\end{proposition} 

Given $\vp\colon X\to [0,\infty]$, let
\begin{equation}\label{def-N}
                      N_\vp:=\inf\bigl\{u\in \N^+\colon u\ge \vp\bigr\} \und \hat N_\vp:=\widehat{N_\vp}. 
\end{equation} 
 Proposition \ref{nearly-proper} immediately yields  the following.

\begin{proposition}\label{N-proper}
\begin{itemize} 
\item[\rm (i)] 
$N_\vp\le R_\vp$, and   $N_\vp$ is the smallest majorant of $\vp$~in~$\N^+$.
\item[\rm (ii)] 
If $ \vp_1,\vp_2,\dots\colon X\to [0,\infty]$ and 
  $\vp_m\uparrow \vp$, then     $      N_{\vp_m}\uparrow N_\vp$ and $\hat N_{\vp_m}\uparrow \hat N_\vp$.
\item[\rm (iii)] If $\vp$ is finely lower semicontinuous, then $N_\vp=R_\vp$.   
\end{itemize}
\end{proposition}

For all functions $\vp\colon X\to  [0,\infty]$ and $x\in X$, let 
\begin{eqnarray*} 
      M_\vp(x)&:= &\sup\bigl\{\int^\ast \vp\,d\mu\colon \mu\in \M_x(\W)\bigr\},\\
      M_\vp'(x)&:= &\sup\bigl\{\int^\ast \vp\,d\vx^\vc\colon x\in V\in \U_c\bigr\},\\
      M_\vp''(x)&:= &\sup\bigl\{\int^\ast \vp\,d\vx^K\colon K\mbox{ compact   in }X\setminus \{x\}\bigr\},\\
      M_\vp'''(x)&:= &\sup\bigl\{\int^\ast \vp\,d\vx^A\colon A\in \B, \ A\mbox{ finely closed},\  x\notin A\}, 
\end{eqnarray*} 
where we may replace the upper integrals by integrals, if $\vp$ is $\bup$-measurable.

\begin{proposition}\label{M-equiv}
Let $\vp$ be a positive numerical function on $X$. Then
\begin{equation}\label{M-form-equiv}
M_\vp'=M_\vp''=M_\vp'''\le M_\vp.
\end{equation} 
If $\vp$ is $\bup$-measurable, then   $\vp\vee M_\vp'=M_\vp$.   
\end{proposition}

\begin{proof} Of course, $M_\vp'\vee M_\vp''\le M_\vp'''\le M_\vp$. Let us fix $x\in X$.

Let  $A$ be a~finely closed Borel set, $x\notin A$,  and $b< 
\int^\ast \vp\,d\vx^A$. By~\cite[VI.4.6]{BH}, $\vx^A$~is supported by $A$. So there exists a compact $K$ in $A$
such that $\int^\ast 1_K \vp\,d\vx^A>b$. By \cite[VI.9.4]{BH}, $1_K\vx^A\le \vx^K$,    hence
 $b<\int^\ast 1_K \vp\,d\vx^A\le \int^\ast \vp\,d\vx^K\le M_\vp''(x)$. Thus $M_\vp'''(x)\le M_\vp''(x)$.

Next let $K$ be a compact set in $X\setminus \{x\}$, $b<c< \int^\ast \vp\,d\vx^K$ and 
 $p\in \Px$, $p>0$. Then there exists $m\in\nat$ such that $\vp':=\vp\wedge (mp)$ satisfies
$\int^\ast \vp'\,d\vx^K>c$. Since $q:=mp$ is a~potential, there exists, by \cite[II.5.2]{BH},
 a relatively compact open neighborhood  $U$ 
of $\{x\}\cup K$ such that $\int q\,d\vx^\uc<c-b$. Let us define $V:=U\setminus K$ and 
$\nu:=1_\uc \vx^\vc$. By \cite[VI.9.4]{BH},  
\begin{equation*}
\nu\le \vx^\uc \und                     \vx^K=1_K\vx^\vc +\nu^K.
\end{equation*} 
Therefore
\begin{equation*} 
 \int q\,d\nu^K\le \int q\,d\nu\le  \int q\,d\vx^\uc<c-b \mbox{ \  and \ } 
c< \int^\ast \vp'\,d\vx^K\le \int^\ast\vp \,d\vx^\vc+ \int q\,d\nu^K .
\end{equation*}  
So $b<\int^\ast \vp \,d\vx^\vc $, and we conclude that $M_\vp''(x)\le M_\vp'(x)$ completing the proof of 
the equalities in (\ref{M-form-equiv}).

Finally, we suppose that $\vp$ is $\bup$-measurable and fix $\mu\in \M_x(\W)$.
Let us  assume for the moment that $\int \vp\,d\mu<\infty$.
There exist positive Borel measurable functions $f,g$ on $X$
such that $f\le \vp \le g$ and 
\begin{equation}\label{fug}
\int f \,d\mu=\int \vp\,d\mu=\int g\,d\mu.
\end{equation} 
Using the integral representation (\ref{int-rep}), we see that $\int f\,d\nu=\int g\,d\nu$ 
for $\rho$-a.e.\ $\nu\in E$, and hence 
\begin{equation*}
                  \int f\,d\nu=\int \vp\,d\nu\le \vp(x)\vee M_\vp'''(x)\qquad\mbox{  for $\rho$-a.e.\ $\nu\in E$}.
\end{equation*} 
  Thus $\int \vp\,d\mu=\int f\,d\mu\le \vp(x)\vee M_\vp'''(x)$, by (\ref{int-rep}) and (\ref{fug}). 
In the general case, we  apply the previous considerations to the functions $\vp\wedge (mu_0)$, $m\in\nat$,
and   let~$m\to\infty$.
\end{proof}

\begin{corollary}\label{muAKV-equiv}
Let  $u $ be a positive numerical function on $X$ and $x\in X$.
Then the following properties are equivalent:
\begin{itemize}
\item[\rm(i)] 
The function $u$ is nearly hyperharmonic.
\item[\rm(ii)] 
For every subset $A$ of $X\setminus\{x\}$, $\int^\ast u\,d\vx^A\le u(x)$.
\item[\rm(iii)]
For every compact $K$ in $X\setminus \{x\}$, $\int^\ast u\,d\vx^K\le u(x)$.
\end{itemize} 
If $u$ is $\bup$-measurable, then these properties hold if and only if $\int u\,d\mu\le u(x)$
for every $\mu\in\M_x(\W)$.
\end{corollary}

So our (positive) 
nearly hyperharmonic functions are functions which in \cite{beznea-boboc-feyel, beznea-boboc-book,
feyel-rep,feyel-fine, moko-bourbaki, moko-ens-compacts} (mostly assuming additional measurability properties)
are called strongly supermedian.

\section{Identity of $M_\vp$ and $N_\vp$, Mertens' formula}

In this section, we shall give a fairly straightforward proof for the following result.  

\begin{theorem}\label{main-result}
For every $\bup$-measurable numerical function $\vp\ge 0$ on $X$,
\begin{equation*} 
                    M_\vp=N_\vp=\vp\vee \hat N_\vp.  
\end{equation*} 
In particular, $M_\vp$ is the smallest nearly hyperharmonic majorant of $\vp$,
and $M_\vp$, $N_\vp$  are {\rm(}at least{\rm)} ``as measurable as $\vp$'', that is, if $\A$
 is any $\sigma$-algebra
on $X$ such that $\B\subset \A\subset \bup$ and $\vp$ is $\A$-measurable,
then $M_\vp$, $N_\vp$ are $\A$-measurable. 
\end{theorem} 

\begin{remark}
{\rm  
In a more general setting, this has been shown by different methods 
 for the smaller class of functions   $\vp\ge 0$ which are
nearly Borel measurable or, slightly more general, nearly analytic  (see \cite{mertens, feyel-fine,
beznea-boboc-feyel, beznea-boboc-book}).
} 
\end{remark}

The following simple possibility of replacing $\vp$ by a smaller function when dealing with envelopes such as $R_\vp$ and $N_\vp$
will be useful.

\begin{proposition}\label{replace-vp}
 Let $\F$ be a convex cone of numerical functions on a set $Y$ and $f_0\in \F$, $0<f_0<\infty$. For every 
numerical function $\vp\ge 0$ on $Y$, let 
\begin{equation*}
                                  F_\vp:=\inf \{ f\in \F\colon f\ge \vp\}.
\end{equation*} 
 Then $F_{\vp1_{X\setminus A}}=F_\vp$ for every numerical function   $\vp\ge 0$ on $Y$ and every $A\subset X$ such that   
$\a\vp\le F_\vp\wedge (Mf_0)$ on~$A$  for some $\a,M\in (1,\infty) $. 
\end{proposition} 

\begin{proof} 
Clearly, it suffices to consider the case, where $\vp$ is not identically zero on~$A$. 
Trivially, $u:=F_{\vp 1_{X\setminus A}}\le F_\vp$. For the reverse inequality let 
$\ve>0$, $ v:=u+\ve f_0$, 
\begin{equation*} 
\b:=\inf\{b\in (0,\infty)\colon  \vp\le b v\mbox{ on }A\} \und \g:=1\vee \b.
\end{equation*} 
Then $0<\b\le M/\ve$ and $\vp \le \g v  $ on $X$. Hence $(\b/\a) v(x)<\vp (x)$
for some $x\in A$ and $F_\vp\le \g v$, which leads to $\b v(x)<\a \vp(x)\le F_\vp(x) \le \g v(x)$.
Thus $\g=1$, $v\ge F_\vp$, and  the proof is completed letting $\ve\to 0$.
\end{proof}

Let  $\vp\colon X\to [0,\infty]$ be $\bup$-measurable. Since  $\vp\le N_\vp$ and $N_\vp\in \N^+$,  we 
obtain, by Corollary \ref{muAKV-equiv}, that 
\begin{equation}\label{trivial}
                                                  M_\vp\le M_{N_\vp}\le N_\vp.
\end{equation} 
To prove the reverse inequality we start considering the case, where $\vp$ is  upper semicontinuous  
and $\mathcal P$-bounded. We first recall the following (\cite[Corollary 1.2.2]{H-course}).

\begin{proposition}\label{R-psi}
For all upper semicontinuous $\mathcal P$-bounded positive 
functions $\psi, \psi_1,\psi_2,\dots$ on $X$ the following holds:
\begin{itemize}
\item
The function $R_\psi$ is upper semicontinuous. It is harmonic on $X\setminus \supp(\psi)$. 
\item
If $\psi$ is continuous, then $R_\psi$ is continuous.
\item
If $\psi_m\downarrow \psi$, then $R_{\psi_m}\downarrow R_\psi$.
 \end{itemize}
\end{proposition} 

The following consequence of the theorem of Hahn-Banach is known
in more general situations (see e.g.\ \cite[p.\ 226]{moko-stresa}).  
For the convenience of the reader we include a~complete proof.

\begin{proposition}\label{R-psi-max}
  Let $\psi\ge 0$ be upper semicontinuous and $\mathcal P$-bounded. 
Then, for every $x\in X$, there exists $\mu\in\M_x(\W)$ such that $ \int \psi \,d\mu=R_\psi(x)$.
\end{proposition} 

\begin{proof} (a) Let $x\in X$ and $\vp\in\C_{\mathcal P}(X)$, $\vp\ge 0$. Since the mapping $f\mapsto R_{f^+}(x)$
is sublinear on $\C_{\mathcal P}(X)$,  there exists a linear form $\mu$ on $\C_{\mathcal P}(X)$ such that
\begin{equation*}
                                                \mu(\vp)=R_\vp(x) \und \mu(f)\le R_{f^+}(x) \mbox{ for every }f\in \C_{\mathcal P}(X). 
\end{equation*} 
If $f\in \C_{\mathcal P}(X)$ and $f\le 0$, then $\mu(f)\le R_0(x)=0$. Therefore $\mu$ is a measure on $X$. Of course,
$\mu(p)\le p(x)$ for every $p\in \Px$, and hence $\mu\in \M_x(\W)$.

(b) There exist $\vp_m\in \C_{\mathcal P}(X)$ such that $\vp_m\downarrow \psi$. By (a), for every $m\in\nat$, 
there exists a measure $\mu_m\in \M_x(\W)$ such that $\mu_m(\vp_m)=R_{\vp_m}(x)$. 
We may (passing to a subsequence) assume without loss
of generality that the sequence $(\mu_m)$ converges to a measure $\mu\in\M_x(\W)$ (that is,
$\lim_{m\to\infty} \mu_m(f)=\mu(f)$ for every $f\in \C_{\mathcal P}(X)$). Then, for every $k\in\nat$,
\begin{equation*}
   R_\psi(x)=\lim_{m\to\infty} R_{\vp_m}(x)=\lim_{m\to\infty} \mu_m(\vp_m)\le \lim_{m\to\infty} \mu_m(\vp_k)=\mu(\vp_k).
\end{equation*} 
Letting $k\to \infty$, we get $R_\psi(x)\le \int \psi\,d\mu$. Trivially     
$\int \psi\,d\mu\le \int R_\psi \,d\mu \le R_\psi(x)$.
\end{proof} 

\begin{corollary}\label{main-psi}
  Let $\psi\colon X\to [0,\infty]$ be upper semicontinuous and  $\mathcal P$-bounded.  Then 
\begin{equation*} 
  M_\psi=N_\psi= R_\psi=\psi\vee \hat R_\psi . 
\end{equation*} 
\end{corollary} 

\begin{proof} By Proposition  \ref{R-psi-max}, $R_\psi\le M_\psi$. 
By (\ref{trivial}) and Proposition \ref{N-proper},  $M_\psi\le N_\psi \le R_\psi$. Therefore
\begin{equation}\label{essential}
M_\psi=N_\psi= R_\psi.
\end{equation} 
To complete the proof it suffices to show that $N_\psi=\psi\vee {\hat N_\psi}$.

To that end let us consider $x\in X$ such that $N_\psi(x)>\psi(x)$. Let $V_m$ be open 
neighborhoods of $x$ such that $V_m\downarrow \{x\}$. Then the functions 
$\psi_m:=\psi 1_{X\setminus V_m}$ are upper semicontinuous and $\psi_m\uparrow \vp:=\psi 1_{X\setminus \{x\}}$. 
Hence $N_{\psi_m}\uparrow N_\vp=N_\psi$  and $\hat N_{\psi_m} \uparrow 
\hat N_\psi$, 
by Propositions~\ref{nearly-proper} and \ref{replace-vp}. 

For every $m\in\nat$, the function  $R_{\psi_m}$ is harmonic on $V_m$, by Proposition \ref{R-psi},
and hence  $\hat R_{\psi_m}(x)=R_{\psi_m}(x)$. So, by (\ref{essential}) (applied to $\psi_m$),     
\begin{equation*} 
\hat N_\psi(x)=\lim\nolimits_{m\to\infty} \hat N_{\psi_m} (x)=\lim\nolimits_{m\to\infty}  N_{\psi_m} (x)=N_\psi(x).
\end{equation*} 
\end{proof}

For every $\vp\colon X\to [0,\infty]$, let $\Psi_\vp$ denote the set of all 
  bounded upper semicontinuous functions $0\le \psi\le \vp$ with compact 
support in $\{\vp>0\}$.
We are now able to prove even  more than announced in Theorem \ref{main-result}.

\begin{theorem}\label{main-theorem} 
Let $\vp\colon X\to [0,\infty] $ be  $\bup$-measurable.  Then    
\begin{equation}\label{main-formula}
                 M_\vp=N_\vp= \vp\vee \hat N_\vp=\sup\{N_\psi\colon \psi\in \Psi_\vp\},
\end{equation} 
  and there is an increasing sequence $(\psi_m)$ in $\Psi_\vp$ such that 
\begin{equation}\label{n-formula}
               N_\vp=\vp\vee \sup\nolimits_{m\in\nat} R_{\psi_m}
                        =\vp\vee \sup\nolimits_{m\in\nat} \hat R_{\psi_m}. 
\end{equation} 
If $\vp\le s$ for some  $s\in\W\cap \C(X)$,  
then $M_\vp$ is harmonic on any open set,  where $M_\vp\ge  \a\vp$ for some $\a>1$. 
\end{theorem} 

\begin{proof} Clearly, $M_\vp=\sup\{M_\psi\colon \psi\in \Psi_\vp\}$, where
$M_\psi=N_\psi =R_\psi=\psi\vee \hat R_\psi$ for every $\psi\in \Psi_\vp$,   by Corollary \ref{main-psi}.
 Since $(\vp\wedge n) 1_{\{x\}} \in\Psi_\vp$, $x\in X$, $n\in\nat$,
we obtain that
\begin{equation}\label{vp-first}
                M_\vp=\sup\{N_\psi\colon \psi\in \Psi_\vp\}=\vp\vee \sup \{{\hat N_\psi}\colon \psi\in \Psi_\vp\}
       \le \vp\vee {\hat N_\vp }\le N_\vp.
\end{equation} 
In particular, $M_\vp$ is $\bup$-measurable.    
By  \cite[I.1.7]{BH}, there is an increasing sequence~$(\psi_m)$ in~$\Psi_\vp$ such that   
$\sup\nolimits_{m\in\nat} \hat R_{\psi_m}= \sup \{\hat R_\psi\colon \psi\in \Psi_\vp\}$. 
Then $\vp\vee  \hat R_{\psi_m}\uparrow M_\vp$ as $m\to \infty$.

We now claim that  
\begin{equation}\label{main-work}
M_\vp\in \N^+,
\end{equation} 
and therefore $ N_\vp\le M_\vp$. 
Having (\ref{vp-first}) this implies that  (\ref{main-formula}) and (\ref{n-formula}) hold.

So let $x\in V\in\U_c$, $\mu:=\vx^\vc$. To show that $\int M_\vp\,d\mu\le M_\vp(x)$  
we may assume that $\int \vp\,d\mu<\infty$, since otherwise $M_\vp(x)=\infty$. 
 Let $a<b<\int M_\vp\,d\mu$. Then there exist $m\in\nat$ and $\psi\in \Psi_\vp$ such that 
\begin{equation*} 
b<\int \vp\vee\hat R_{\psi_m}\,d\mu \und \int (\vp-\psi)\,d\mu < b-a.
\end{equation*} 
 Of~course, we may assume that $\psi_m\le \psi$. Since trivially
$\vp\vee \hat R_\psi-\psi\vee \hat R_\psi\le \vp-\psi$, we obtain that
\begin{equation*}
        b<\int \vp \vee \hat R_\psi \,d\mu\le \int \psi\vee \hat R_\psi\,d\mu+(b-a),
\end{equation*} 
and hence
\begin{equation*}
                a<\int \psi\vee \hat R_\psi\,d\mu=\int R_\psi\,d\mu\le R_\psi(x) \le M_\vp(x).
\end{equation*} 
Thus $\int M_\vp\,d\mu\le M_\vp(x)$ proving (\ref{main-work}).

Finally, suppose that $\vp\le w$ for some $w\in \W\cap \C(X)$ and let $U$ be an open set, where $\vp$
vanishes. Then all functions  $g_m:=R_{\psi_m}|_U$  are harmonic on $U$, by Proposition~\ref{R-psi}, and 
hence $M_\vp|_U=\sup g_m$ is harmonic on $U$. An application of Proposition~\ref{replace-vp}
completes the proof.
\end{proof} 

\begin{remark}\label{axiom-polarity}
{\rm
 If  the semipolar sets $\{\widehat{\inf \V}<\inf \V\}$, $\V\subset \W$, are polar
(axiom of polarity, Hunt's hypothesis (H)) and $\vp\colon X\to [0,\infty]$ is $\tilde \B$-measurable, 
then by \cite[Theorem 2.2]{reduced-jensen} and (\ref{n-formula}) there exists an increasing sequence 
 $(\psi_n)$ in $\Psi_\vp$ such that
\begin{equation} \label{R=N}
                       R_\vp=\vp\vee \sup\nolimits_{n\in\nat} R_{\psi_n}=N_\vp.
\end{equation} 
By \cite[Theorem 6.3]{beznea-boboc-feyel}, (\ref{R=N}) holds even without assuming the axiom of polarity.
}
\end{remark}

\section{Application to Jensen measures}\label{jensen-appl}

In this section,  let us suppose that $(X,\W)$ is a harmonic space, that is,
the harmonic measures $\mu_x^V=\vx^\vc$,   $V$ relatively compact 
open in $X$, $x\in X$, are supported by the boundary $\partial V$ of $V$.

Given an open set $U$  in $X$,  let $\h(U)$  denote the set of all 
hyperharmonic functions on $U$, that is, lower semicontinuous numerical functions $w>-\infty$ on $U$
such that $\int w\,d\vx^\vc\le w(x)$  for all open $V$, which are relatively compact in $U$, and  $x\in V$.

Given $x\in U$, 
let $\J_x(U)$ denote the set of all Jensen measures for~$x$ with respect to~$U$,
that is, measures~$\mu$ with compact support in~$U$ satisfying 
\begin{equation}\label{def-jensen}
                                 \int w\,d\mu\le w(x) \qquad \mbox{ for every }w\in \h(U). 
\end{equation} 
In fact, it suffices to know (\ref{def-jensen}) for all $w\in \h(U)\cap \C(U)$, since every $w\in \h(U)$ 
is an  increasing limit of functions in $ \h(U)\cap \C(U)$. 

Since $\W=\{w\in\h(X)\colon w\ge 0\}$ and $\h(X)|_U\subset \h(U)$, we have 
\begin{equation*} 
\vx\in \J_x(U)\subset \J_x(X)\subset \M_x(\W), \qquad x\in U
\end{equation*} 
(where we consider   measures in $\J_x(U)$ as measures on $X$).
It will be convenient to introduce also the union  $\J_x'(X)$ of  all $\J_x(U)$, $U$ open relatively compact in $X$,

$x\in U$ (see~\cite{HN-jensen} for properties implying that $\J_x'(X)=\J_x(X)$).

Finally, for every locally lower bounded function $\vp$ on $X$ which is 
$\bup$-measurable, we define functions $J_\vp$ and $J_\vp'$ on $X$ by 
\begin{equation*}
                             J_\vp(x):=\sup\bigl\{\int \vp\,d\mu\colon \mu\in \J_x(X)\bigr\}
\und                    J_\vp'(x):=\sup\bigl\{\int \vp\,d\mu\colon \mu\in \J_x'(X)\bigr\}.
\end{equation*} 

If $\vp\ge 0$, then obviously $M_\vp'\le J_\vp'\le J_\vp\le M_\vp$. Therefore Proposition \ref{M-equiv} and Theorem \ref{main-theorem}
immediately yield  the following.

\begin{theorem}\label{jensen-theorem}
Let $\vp$ be a positive $\bup$-measurable numerical function on $X$. Then 
\begin{equation*}
                                  J_\vp=J_\vp'=  \vp\vee  M_\vp'=M_\vp=N_\vp=\vp\vee \hat N_\vp.
 \end{equation*} 
  In particular, $J_\vp$ is  Borel measurable if $\vp$ is Borel measurable. 
\end{theorem} 

Similarly as in \cite{reduced-jensen} we may now extend this result to functions $\vp$ which
are not necessarily positive. To that end let $\N$ denote the set of all nearly hyperharmonic 
functions on $X$, that is, locally lower bounded functions $w\colon X\to ]-\infty,\infty]$ such that 
$\int^\ast w\,d\vx^\vc \le w(x)$ for all $x\in X$ and relatively compact open neighborhoods $V$ of $x$.
We immediately get  the following generalization of  Proposition \ref{nearly-proper}.

\begin{proposition}\label{nearly-proper-general}
The set $\N$  of all  nearly hyperharmonic functions on $X$ has the following properties:
\begin{itemize} 
\item[\rm (i)] 
$\N$ is a convex cone containing $\h(X)$. 
\item[\rm (ii)] 
For every $u\in\N$, $\hat u=\hat u^f\in\h(X)$.
\item[\rm (iii)] 
If $(u_m)$ is a sequence in $\N$ and $u_m\uparrow u$, then $u\in\N$ and $\hat u_m\uparrow \hat u$.
\item[\rm(iv)] For every subset $\V$ of $\N$ which is locally lower bounded, $\inf \V\in \N$. 
\end{itemize} 
\end{proposition}

Extending the definitions of 
$J_\vp$, $J_\vp'$,    and $M_\vp'$ in an obvious way, we get the following.

\begin{corollary}\label{jensen-reduced-h}
Let $\vp$ be a locally lower bounded $\bup$-measurable 
numerical function on $X$ such that
$\vp+h\ge 0$ for some harmonic function $h $ on $X$. 
 Then  
\begin{equation*}
                                  J_\vp=J_\vp'= \vp\vee M_\vp'=N_\vp=\vp\vee \hat N_\vp.
 \end{equation*} 
  In particular, $J_\vp$ is  Borel measurable if $\vp$ is Borel measurable.
\end{corollary} 

\begin{proof}
It suffices to observe that $\vp+h$ is $\bup$-measurable 
and obviously  $    J_\vp=J_{\vp+h}-h$, $    J_\vp'=J_{\vp+h}'-h$, $    M_\vp'=M'_{\vp+h}-h$,  
and $    N_\vp=N_{\vp+h}-h$.
\end{proof} 

Localizing this result we may deal with functions $\vp$ which are locally lower bounded.

\begin{corollary}\label{jensen-reduced-hn}
Let $\vp$ be a locally lower bounded $\bup$-measurable numerical function on $X$ such that,
for every relatively compact open set $U$ in $X$, there exists a~harmonic function $h$ on $X$
with $\vp+h\ge 0$ on $U$ . Then  
\begin{equation*}
                                  J_\vp'= \vp\vee M_\vp'=N_\vp=\vp\vee \hat N_\vp.
 \end{equation*} 
  In particular, $J_\vp'$ is  Borel measurable if $\vp$ is Borel measurable.
\end{corollary}

\begin{proof} Let $U_n$  be relatively compact open sets in $X$  such that 
$U_n\uparrow X$ as $n\to \infty$. 
For every $n\in \nat$,  we apply Corollary \ref{jensen-reduced-h} to the harmonic space~$(U_n, \h^+(U_n))$ 
and obtain that, for~$x\in U_n$,
\begin{multline*} 
             \sup\{ \int^\ast \vp\,d\mu\colon \mu\in \J_x(U)\}
          =     
  \vp(x)\vee       \sup\{ \int^\ast \vp\,d\vx^\vc\colon x\in V\in \U_c,\ \ov V\subset U\}\\
        =  \inf\{w(x)\colon w\mbox{ nearly hyperharmonic on $U_n$, $w\ge \vp$ on $U_n$}\} = :v_n(x),
\end{multline*} 
where $v_n(x)=\vp(x)\vee \hat v_n(x)$. 
Defining $v_n(x):=\vp(x)$, $x\in X\setminus U_n$, we easily see that the sequence $(v_n)$ is increasing 
to a nearly hyperharmonic function $v$ on $X$, where  $v=N_\vp=\vp\vee \hat N_\vp$,
by Proposition \ref{nearly-proper-general}. The proof is completed letting $n\to \infty$. 
\end{proof}

\begin{remark} {\rm
By Remark \ref{axiom-polarity}, the  results in this Section imply the results in \cite[Section 3]{reduced-jensen}.
}
\end{remark}

\section{Some improvement of the measurability} \label{fine-imp}

Let us now return to the general situation of an arbitrary balayage space $(X,\W)$. 
Sometimes we can say a bit more about the measurability of $N_\vp$ (and hence 
on the measurability of $J_\vp$ in Section \ref{jensen-appl}). 

We  recall that a set $A$ in $X$ is called \emph{thin at $x\in X$} if $\hat \ve_x^A\ne \vx$.
It is \emph{totally thin} if it is thin at every $x\in X$. Every totally thin set is finely closed
and contained in a totally thin Borel set. 
A \emph{semipolar} set is a countable union of totally thin sets. 

So the $\sigma$-algebra $\B^f$  of all finely Borel subsets of $X$  
(that is, the smallest \hbox{$\sigma$-algebra} on $X$ containing all finely open 
sets) contains all semipolar sets. By \cite[VI.5.16]{BH}),  for every $B\in\B^f$, 
there are $B_1,B_2\in \B$ such that $B_1\subset B\subset B_2$ and 
$B_2\setminus B_1$ is semipolar. 
Thus $\B^f$ is the smallest $\sigma$-algebra
containing $\B$ and all semipolar sets. In particular, $\tilde \B\subset \B^f$.

\begin{example}{\rm   
Suppose for the moment that $X=\real^d\times \real$, $d\ge 1$,
and $\W$ is the set of all positive hyperharmonic functions associated 
with the heat equation on~$\real^{d+1}$. Let $S$ be \emph{any} subset of $\real^d\times \{0\}$.
Then $S$ is semipolar;  it is polar if and only if~$S$ has outer $d$-dimensional
Lebesgue measure zero. The function $u:=1_S+1_{\reald\times (0,\infty)}$ is nearly hyperharmonic,
$\B^f$-measurable, and $\{\hat u<u\}=S$.
}\end{example}

\begin{proposition}\label{semipolar-polar}
Let $S\in \bup$ be semipolar. Then there exists  a sequence~$(K_n)$ of compacts in $S$
such that the set $S\setminus \bigcup_{n=1}^\infty K_n$ is polar. In particular, $S\in \tilde \B$.
\end{proposition} 

\begin{proof} By \cite[Theorem 1.5]{S-negligible} (which holds as well for balayage spaces),
there exists a~measure $\mu$ on $X$ such that $\mu^\ast(B)>0$ for every 
subset $B$ of $S$ which is not polar. There exists a subset $S_0\in \bu$
of $S$ such that the set $P_0:=S\setminus S_0$ is polar. Moreover,
there exists a sequence $(K_n)$ of compacts 
in $S_0$ such that $P_1:=S_0\setminus \bigcup_{n=1}^\infty K_n$ is a~$\mu$-null set,
and hence polar. Since $P_0\cup P_1$ is polar, the proof is finished.
\end{proof}

The equivalence in the following proposition is of no interest, 
if we know that  $u=\inf \V$, $\V\subset \W$, 
since then $u$ is obviously finely upper semicontinuous and the set $\{\hat u<u\}$
is semipolar by \cite[VI.5.11]{BH}.

\begin{proposition}\label{nh-semipolar}
For every $u\in\N^+$  the following statements are equivalent:
\begin{itemize}
\item [\rm (i)] The set $\{\hat u<u\}$ is semipolar.
\item [\rm (ii)] The function $u$ is finely Borel measurable.
\end{itemize} 
\end{proposition} 

\begin{proof} (i)\,$\Rightarrow$\,(ii):
For every $t>0$, the set $\{u\ge t\}$ is the union of $\{\hat u\ge t\}\in \B$ and the 
semipolar set $\{u\ge t>\hat u \}$.

 (ii)\,$\Rightarrow$\,(i):
There is a semipolar Borel set $S_0$ such that the function $u_0:=u1_{X\setminus S_0}$
is $\B$-measurable. Suppose that the set $\{\hat u<u\}$ is not semipolar.
Then the Borel set $A:=\{\hat u<u\}\setminus S_0$ is not semipolar. So, by \cite[VI.8.9]{BH},
there is a measure $\mu\ne 0$ on $X$ such that $\mu(X\setminus A)=0$ and $\mu$
does not charge semipolar sets. There exist functions $\psi_m\in \Psi_{u_0}$ such that $\psi_m\uparrow u_0$
outside a $\mu$-null set $B\in \B$.   By Corollary \ref{main-psi},
\begin{equation*} 
               \psi_m\le R_{\psi_m }=N_{\psi_m }\le N_u=u, \qquad m\in\nat.
\end{equation*} 
Hence $\{\sup_{m\in\nat} R_{\psi_m}<u\}\subset B\cup S_0$. Further, 
 the union $S_1$ of all sets $\{\hat R_{\psi_m}<R_{\psi_m}\}$, $m\in\nat$, is semipolar,
and we obtain that
 \begin{equation*}
      \hat u \ge  \sup\nolimits_{m\in\nat} \hat R_{\psi_m}=
 \sup\nolimits_{m\in\nat}  R_{\psi_m}=u \on X\setminus (B\cup S_0\cup S_1).
\end{equation*} 
Thus $A\subset B\cup S_1$ and $\mu(X)=\mu(A)=0$,    a contradiction. 
\end{proof}

\begin{corollary}\label{meas-improv}
If  $\vp\colon X\to [0,\infty]$  is $\B^f\cap \bup$-measurable, then the function~$N_\vp$ is $\bp$-measurable.
\end{corollary} 

\begin{proof} By Theorem \ref{main-theorem}, the function $u:=N_\vp$ is $ \B^f\cap \bup$-measurable. 
Let $t\in \real$ and $S:=\{u\ge t> \hat u\}$. Then $S\in \bup$ and $S$ is  semipolar, 
by Proposition \ref{nh-semipolar}. So $S\in\bp$, by Proposition \ref{semipolar-polar},
and   $\{u\ge t\}= \{\hat u\ge t\}\cup S\in \bp$.
\end{proof} 

Further, let $\A(X)$ denote the set of all numerical functions $\vp$ on $X$   
having the following  property: For every $t\in \real$, there exists an
analytic set $A$ in $X$ such that the set $\{\vp\ge t\}\bigtriangleup A$
is semipolar. 
 By the discussion preceding Proposition \ref{nh-semipolar}, $\vp\in \A(X)$ for every
 finely Borel measurable function $\vp$.

\begin{proposition}\label{fusc}
Let $\vp$ be a positive function in $\A(X)$ which is $\bup$-measurable. 
Then $N_\vp$ is finely upper semicontinuous and $\bp$-measurable.
\end{proposition} 

\begin{proof}  By Theorem \ref{main-theorem}, we know that $u:=N_\vp\in \A(X)$ and 
$u$ is $\bup$-measurable.
Let $x\in X$ and $u(x)<a$. We claim that $\{ u<a\}$ is a fine neighborhood 
of~$x$. Indeed, suppose the contrary.
Then the set $A:=\{u\ge a\}$ is not thin at $x$.
Let~$A'$ be an analytic set such that $A'\subset A$ and   $A\setminus A'$ is
semipolar.  We fix $\eta\in (0,1)$ such that $u(x)<a\eta^2$. Since $V:=\{u_0(x)> \eta u_0\}$
is a neighborhood of $x$, we know, by   \hbox{\cite[VI.4.2]{BH}}, that either the analytic set $A'\cap V $ or
the semipolar set $S:=(A\setminus A')\cap V$ is not thin at $x$.

If $A'\cap V$ is not thin at $x$, then,   by \cite[VI.1.10 and 1.3.5]{BH} ,  there is a compact~$K$ in~$A'\cap V$ such that 
$R_{u_0}^K(x)>\eta u_0(x)$.
By definition of semipolar sets, $S$ is the union of totally thin sets $T_m$, $m\in\nat$.
 By \cite[VI.5.7]{BH}, 
the union of finitely many totally thin sets is totally thin. Hence  we may assume without loss of generality that 
$T_m\uparrow S$ as~$m\to\infty$. If $S$ is not thin at $x$, we then obtain, 
 by \cite[VI.1.7]{BH},  that $R_{u_0}^{T_m}(x)> \eta u_0(x)$ for some $m\in\nat$.

Thus, in any case, there exists a finely closed 
set $F$   such that 
\begin{equation*}
          F\subset A\cap V\in \bup \und                         R_{u_0}^F(x)>\eta u_0(x).
\end{equation*} 
 Since $u\ge a> a \eta u_0(x)\inv  u_0$ on~$A\cap V$ 
and $\vx^F(X\setminus (A\cap V))=0$, by \cite[VI.4.6]{BH}, 
we conclude that
\begin{equation*} 
u(x)\ge \int u\,d\vx^F\ge a\eta u_0(x)\inv  \int  u_0\,d\vx^F
=  a \eta u_0(x)\inv  R_{u_0}^F(x)\ge  a \eta^2  > u(x),
\end{equation*} 
a contradiction. By Corollary \ref{meas-improv}, the proof is finished.
\end{proof} 

\begin{corollary}\label{A-B} Let $\vp\colon X\to [0,\infty]$ be such that, for every $t>0$,
there exists an analytic set $A$ in $X$ such that the set $\{\vp\ge t\}\bigtriangleup A$ is polar.
Then $N_\vp$ is $\bp$-measurable.
\end{corollary} 


{\small \noindent 
Wolfhard Hansen,
Fakult\"at f\"ur Mathematik,
Universit\"at Bielefeld,
33501 Bielefeld, Germany, e-mail:
 hansen$@$math.uni-bielefeld.de}\\
{\small \noindent Ivan Netuka,
Charles University,
Faculty of Mathematics and Physics,
Mathematical Institute,
 Sokolovsk\'a 83,
 186 75 Praha 8, Czech Republic, email:
netuka@karlin.mff.cuni.cz}

\end{document}